\documentclass[a4paper, 12pt, oneside]{amsart}

\usepackage[top=35mm,bottom=30mm,left=25mm,right=25mm]{geometry}	

\usepackage{mathrsfs}
\usepackage{amsfonts}
\usepackage{amssymb}
\usepackage{amsxtra}
\usepackage{color}
\usepackage{dsfont}
\usepackage[colorinlistoftodos,  textsize=tiny]{todonotes} 					 
\presetkeys{todonotes}{fancyline, color=blue!30}{}

\usepackage[english,polish]{babel}
\usepackage[colorlinks,citecolor=blue,urlcolor=blue,bookmarks=true]{hyperref}
\hypersetup{
pdfpagemode=UseNone,
pdfstartview=FitH,
pdfdisplaydoctitle=true,
pdfborder={0 0 0}, 
pdftitle={Heat content for convolution semigroups},
pdfauthor={Wojciech Cygan, Tomasz Grzywny},
pdflang=en-US
}

\newcommand{\sprod}[2]{\langle {#1}, {#2}\rangle}
\newcommand{\norm}[1]{{\lVert #1 \rVert}}

\newcommand{\RR}{\mathbb{R}}
\newcommand{\PP}{\mathbb{P}}
\newcommand{\EE}{\mathbb{E}}

\newcommand{\IndFun}{\textbf{1}_}

\newcommand{\calR}{\mathcal{R}}

\newcommand{\ud}{\mathrm{d}}
\newcommand{\psiI}{V^{-}(1/t)}

\newcommand{\RInf}{\calR_{\alpha} }

\newcommand{\pl}[1]{\foreignlanguage{polish}{#1}}

\newtheorem{theorem}{Theorem}

\newtheorem{lemma}{Lemma}
\newtheorem{corollary}{Corollary}

\theoremstyle{definition}
\newtheorem{example}{Example}

\title[Generalized heat content]{A note on the generalized heat content\\ for L\'{e}vy processes}

\author{Wojciech Cygan}
\thanks{W.~Cygan was supported by National Science Centre (Poland): grant 2015/17/B/ST1/00062 and by Austrian Science Fund project FWF P24028.}

\author{Tomasz Grzywny}

\address{
		Wojciech Cygan\\
		Instytut Matematyczny\\
		Uniwersytet \pl{Wroc{\lll}awski}\newline
		Pl. Grun\-waldzki 2/4\\
		50-384 \pl{Wroc{\lll}aw}, Poland \newline
		\& Institut f\"{u}r Diskrete Mathematik\\
 		Technische Universit\"{a}t Graz\newline
 		Steyrergasse 30\\
 		8010 Graz, Austria}
\email{wojciech.cygan@uwr.edu.pl}

\address{
	Tomasz Grzywny\\
	 \pl{Wydzia{\lll}} Matematyki\\
	Politechnika \pl{Wroc{\lll}awska}\newline
	Wyb. \pl{Wyspia\'{n}skiego} 27\\
	50-370 \pl{Wroc\l{}aw}, Poland}
\email{tomasz.grzywny@pwr.edu.pl}


\begin{document}
\selectlanguage{english}

\begin{abstract}
Let $\mathbf{X}=\{X_t\}_{t\geq 0}$ be a L\'{e}vy process in $\RR^d$ and $\Omega$ be an open subset of $\RR^d$ with finite Lebesgue measure. The quantity 
$H_{\Omega} (t) = \int_{\Omega}\PP^{x} (X_t\in \Omega )\, \ud x$ 
is called the heat content. 
 In this article we consider its generalized version $H_g^\mu (t) = \int_{\RR^d}\EE^{x} g(X_t)\mu( \ud x )$, where $g$ is  a bounded function and $\mu$ a finite Borel measure. We study its asymptotic behaviour at zero for various classes of L\'{e}vy processes.
\end{abstract}

\maketitle
\section{Introduction}

Let $\mathbf{X}=(X_t)_{t\geq 0}$ be a L\'{e}vy process in $\RR ^d$ with the distribution $\PP$ and such that $X_0=0$. 
We denote by $p_t(\ud x)$ the distribution of the random variable $X_t$ and 
we use the standard notation $\PP^x$ for the distribution related to the process $\mathbf{X}$ started at $x\in \RR^d$. 

Let $\Omega$ be a non-empty open subset of $\RR ^d$ such that its Lebesgue measure $|\Omega|$ is finite. 
We consider the following quantity associated with the process $\mathbf{X}$,
\begin{align}\label{Heat_Content}
H _{\Omega } (t) = \int_{\Omega}\PP^{x} (X_t\in \Omega)\, \ud x =  \int_{\Omega}\int_{\Omega -x}p_t( \ud y)\ud x,\quad t\geq 0.
\end{align} 
Observe that the function $u(t,x) = \int_{\Omega-x}p_t(\ud y)$ is the weak solution of the initial value problem
\begin{align}\label{Heat_eq}
\frac{\partial}{\partial t}u(t,x) &= \mathcal{L}\, u(t,x),\quad t>0,\, x\in \RR^d,\\
u(0,x) &= \IndFun{\Omega}(x)\notag ,
\end{align}
where $\mathcal{L}$ is the infinitesimal generator of the process $\mathbf{X}$, see \eqref{gener}. Therefore, $H_\Omega (t)$ can be interpreted as the amount of \textit{heat} in $\Omega$ if its initial temperature is one whereas the initial temperature of $\Omega^c$ is zero. The quantity $H_\Omega (t)$ is called the heat content. The asymptotic behaviour - as $t$ goes to zero - of the heat content related to the Brownian motion, either on $\RR^d$ or on compact manifolds, were studied in many papers among which \cite{vanDenBerg1_POT}, \cite{vanDenBerg1}, \cite{vanDenBerg2}, \cite{vanDenBerg3}, \cite{vanDenBerg_4}, \cite{vanDenBerg5}. The heat content for the isotropic stable processes in $\RR^d$ was studied in \cite{Valverde1}, see also \cite{Valverde2} and \cite{Valverde3}. The direct forerunner of the present paper is article \cite{CG} where asymptotic behaviour of \eqref{Heat_Content} were found for numerous examples of L\'{e}vy processes.					

In this note we study an extended version of the heat content \eqref{Heat_Content}. Namely,  
for a bounded function $g$ and a finite Borel measure $\mu$, we consider the quantity
\begin{align*}
H ^\mu _{g} (t) = \int_{\RR^d}\EE^{x} g(X_t)\mu( \ud x) =  \int_{\RR^d}v(t,x) \mu( \ud x),\quad t\geq 0.
\end{align*}
Here, the function $v(t,x) = \EE^{x} g(X_t)$ is the weak solution to equation \eqref{Heat_eq} with the initial condition $v(0,x) = g(x).$
Thus, $H^\mu_g(t)$ can be interpreted as the amount of \textit{heat} in the set $\mathrm{supp}(\mu)$ if its initial temperature is governed by the function $g$. On the other hand, the measure $\mu$ can be regarded as the initial distribution on $\RR^d$. Notice that taking $g=\IndFun{\Omega}$ and $\mu( \ud x) = \IndFun{\Omega}\ud x$ we obtain that $H_g^\mu$ is equal to the heat content defined at \eqref{Heat_Content}. 

On the basis of the methods developed in \cite{CG}, we study the asymptotic behaviour of the quantity $H_g^\mu (t)$. We now display the results together with necessary facts and definitions.

\subsection*{Notation.}
By $B_R$ we denote the closed ball $\{x\in \RR^d:\, \norm{x}\leq R \}$ and by $\mathbb{S}^{d-1}$ the unit sphere in $\mathbb{R}^d$. Positive constants are denoted by $c,C,C_1$ etc. 
We write: $f(x)\asymp g(x)$ if there are $c,C>0$ such that $cg(x)\leq f(x)\leq Cg(x)$; $f(x)=o(g(x))$ at $x_0$ if $\lim_{x\to x_0}f(x)/g(x)=0$, and $f(x)\sim g(x)$, as $x\to x_0$, if $\lim_{x\to x_0} f(x)/g(x)=1$. The generalized inverse $V^-$ of the function $V$ is given by $V^-(u) = \inf \{x\geq 0:\, V(x)\geq u\}$. $C_b(\RR^d)$ is the set of all bounded and continuous functions in $\mathbb{R}^d$ whereas $C_0(\RR^d)$ is the set of all continuous functions which vanish at infinity.

\subsection*{Results and basic facts.}
The L\'{e}vy-Khintchine exponent $\psi (x)$ of the L\'{e}vy process $\mathbf{X}$ is given by the formula
\begin{align}\label{charact_expo}
\psi(x) = \sprod{x}{Ax} - i\sprod{x}{\gamma } - 
\int_{\RR ^d}\left(e^{i\sprod{x}{y}} - 1 - i\sprod{x}{y} \IndFun{\{\norm{y} \leq 1\}} \right) \nu(\ud y) ,\ \ x\in \RR ^d,
\end{align}
where $A$ is a symmetric non-negative definite $d\times d$ matrix, $\gamma\in \RR ^d$ and $\nu$ is a L\'{e}vy measure, that is $\nu(\{0\})=0$ and $\int_{\RR ^d} \left( 1\wedge \norm{y}^2 \right)\, \nu(\ud y) <\infty $.

The heat semigroup $\{T_t\}_{t\geq 0}$ related to the L\'{e}vy process $\mathbf{X}$ is 
\begin{align*}
T_tf(x) = \int_{\RR^d} f(x+y)p_t(\ud y),\quad \mathrm{for}\ f\in C_0(\RR^d),
\end{align*}
and the generator $\mathcal{L}$ of the process $\mathbf{X}$ is a linear operator defined by
\begin{align}\label{gener}
\mathcal{L}f = \lim_{t\to 0^+}t^{-1} \left( T_tf - f \right),
\end{align}
with the domain $\mathrm{Dom}(\mathcal{L})$ which is the set of all $f$ such that the right-hand side of \eqref{gener} exists in the sense of uniform convergence.
By \cite[Theorem 31.5]{Sato}, we have $C_0^2(\RR^d)\subset \mathrm{Dom}(\mathcal{L})$ and for any $f\in C_0^2(\RR^d)$ it has the form
\begin{align}\label{Gener_form}
\begin{aligned}
\mathcal{L} f(x) &= \sum_{j,k=1}^d A_{jk}\partial^2_{jk} f(x)+\sprod{\gamma}{\nabla f(x)} \\
&\quad + \int_{\RR^d} \left(f(x+z)-f(x)
-\IndFun{\norm{z}<1}\sprod{z}{\nabla f(x)} \right) \nu(\ud z),\quad x\in \RR^d,
\end{aligned}
\end{align}
where $(A,\gamma ,\nu)$ is the triplet from \eqref{charact_expo}. We refer the reader to \cite[Section 31]{Sato} or \cite[Section 3.3]{Appl} for a detailed discussion on infinitesimal generators of semigroups related to L\'{e}vy processes.

To start our discussion on the small time behaviour of $H_g^\mu(t)$ we make an important observation. 
Let $g$ be a bounded function and $\mu$ a finite Borel measure. We set $\check{\mu }(G) =\mu (-G)$, for any Borel set 
$G\subset \RR^d$ and consider the following convolution
\begin{align*}
r (x) = g\ast \check{\mu }(x) = \int_{\RR^d} g(x+y)\mu(\ud y).
\end{align*}
We can then write
 \begin{align}\label{Generator_formula}
 \begin{split}
 H_g^\mu (t) &=  \int_{\RR^d}\EE^{x} g(X_t)\mu(\ud x) = \int_{\RR^d} \int_{\RR^d} g(y+x)p_t(\ud y)\mu(\ud x)\\
 &= \int_{\RR^d}g\ast \check{\mu }(y)\, p_t(\ud y)  = T_tr(0)
	\end{split}  
  \end{align}
and therefore
\begin{align*}
\lim_{t\to 0^+} t^{-1} \left( H_g^\mu (t) -H_g^\mu (0)\right) = \mathcal{L}r(0),
\end{align*}
whenever $r$ belongs to $\mathrm{Dom}(\mathcal{L})$. Notice that in front of formula \eqref{Generator_formula}, we are rather interested in the pointwise limit in \eqref{gener} instead of the uniform convergence. Thus, for some special classes of L\'{e}vy processes we will weaken the assumption that $r\in \mathrm{Dom}(\mathcal{L})$. This is summarized in Theorem \ref{Thm_X_bdd_variation}.	

 Recall that according to \cite[Theorem 21.9]{Sato} a L\'{e}vy process $\mathbf{X}$ has finite variation on any interval $(0,t)$ if and only if $A=0$ and $ \int_{\norm{x}\leq 1}\norm{x}\nu (\ud x)<\infty $.
In this case the L\'{e}vy-Khintchine exponent has the simplified form
\begin{align}\label{gamma_0}
\psi (x) = i\sprod{x}{\gamma _0} + \int_{\RR^d}\left( 1-e^{i\sprod{x}{y}}\right)\nu (\ud y),\quad \mathrm{with}\ \ 
\gamma _0 = \int_{\norm{y}\leq 1}y\, \nu (\ud y) - \gamma .
\end{align}
Notice that for symmetric L\'{e}vy processes with finite variation we have $\int_{\norm{y}\leq 1}y\, \nu (\ud y) =0$.
\begin{theorem}\label{Thm_X_bdd_variation}
	Let $\mathbf{X}$ be a L\'{e}vy process in $\RR^d$ with the triplet $(0 ,\gamma , \nu)$ and such that
	\begin{align*}
	\int _{\norm{x}<1}\norm{x}^\beta \nu (\ud x)<\infty,\quad \mathrm{for\ some}\ 0\leq \beta <2.
	\end{align*}
	Assume that $g$ is bounded, $\mu$ is finite Borel measure and that the function $r = g\ast \check{\mu}$ is in $C_b(\RR^d)$. We distinguish the following cases.\\
	1.
	$0\leq \beta \leq 1$. In this case we assume that $\gamma_0 =0$ and $|r(x)-r(0)|\leq C\norm{x}^\beta$, for $\norm{x}<1$. Then
	\begin{align*}
	\lim_{t\to 0^+}t^{-1}\left( H_g^\mu (t) -H_g^\mu (0)\right)
	 = \int_{\RR^d} (r(x)-r(0))\nu (\ud x).
	\end{align*}
	2. $1\leq \beta <2$. We consider two cases
	\begin{itemize}
	\item[(i)]
	If $\mathbf{X}$ is symmetric (i.e. $\gamma =0$ and $\nu $ is symmetric) we assume that function $r$ satisfies
	 $|r(x)+r(-x)-2r(0)|\leq C\norm{x}^\beta$, for $\norm{x}<1$. Then
	\begin{align*}
	\lim_{t\to 0^+}t^{-1}\left( H_g^\mu (t) - H_g^\mu (0)\right) = \frac{1}{2} \int_{\RR^d} (r(x)+r(-x)-2r(0))\nu (\ud x).
	\end{align*}
	\item[(ii)]
	If $\mathbf{X}$ is arbitrary, assume that $r$ is differentiable at $0$ and such that\\
	 $|r(x)-r(0)-\sprod{x}{\nabla r(0)}|\leq C\norm{x}^\beta $, for $\norm{x}<1$. Then
	\begin{align*}
	\lim_{t\to 0^+}t^{-1}\left( H_g^\mu (t) - H_g^\mu(0)\right) = 
	\big\langle \gamma  											
				,\, \nabla r(0)\big\rangle
	 +
	\int_{\RR^d} \big( r(x)-r(0)-\sprod{x}{\nabla r(0)}\IndFun{\{\norm{x}\leq 1\}}\big)\nu (\ud x).
	\end{align*}
	\end{itemize}
\end{theorem}

We emphasize that Theorem \ref{Thm_X_bdd_variation} was recently obtained by K\"{u}hn and Schilling \cite{Kuhn} for a wider range of stochastic processes, namely for the class of \textit{rich L\'{e}vy-type} processes, cf. Theorem 3.5 and Theorem 4.1 in \cite{Kuhn}. Our result is expressed in terms of the heat content and extends slightly an admissible classes of functions for L\'{e}vy processes.

The next theorem provides the asymptotic behaviour of the generalized heat content under the assumption that the L\'{e}vy-Khintchine exponent $\psi$ is a multivariate regularly varying function, see condition \eqref{Mult_Reg_Cond} and \cite[Chapter 6]{Resnick} for an elaborate approach.
Recall that in the one-variable case a function $f(r)$ is regularly varying of index $\alpha$ at infinity, denoted by $f\in \RInf$, if for any $\lambda >0$,
$\lim_{r\to \infty}\frac{f(\lambda r)}{f(r)} = \lambda ^\alpha $.
The following property, so-called \textit{Potter bounds}, of regularly varying functions appears to be very useful, see
\cite[Theorem 1.5.6]{bgt}. For every $C > 1$ and $\varepsilon > 0$ there is $x_0=x_0(C,\varepsilon)>0$ such that for all
$ x, y \geq x_0$
\begin{equation}
	\label{eq:14}
	\frac{f(x)}{f(y)}\leq C \max \left\{(x/y)^{\alpha -\varepsilon} , (x/y)^{\alpha +\varepsilon} \right\}.
\end{equation}
For a given function $\psi$ we define the related non-decreasing function $\psi ^*$ by 
\begin{align*}
\psi^*(u) = \sup_{\norm{x}\leq  u} \psi(x) .
\end{align*}

\begin{theorem}\label{Multivariate_Reg_Var_thm}
Let $\beta \in [1,2)$ be fixed. Let $\mathbf{X}$ be a symmetric L\'{e}vy process in $\RR^d$ with 
the L\'{e}vy-Khintchine exponent $\psi$.
We assume that
\begin{align}\label{PsiStar}
\psi(x) \asymp \psi^*(x),\quad \mathrm{for}\ \norm{x}\geq 1.
\end{align}
Suppose that there is a function $V\in \RInf$ with $\alpha \in (\beta ,2]$ and a continuous function\\ $\Lambda \colon \mathbb{S}^{d-1}\to (0,\infty)$ such that  
\begin{align}\label{Mult_Reg_Cond}
\lim_{s\to \infty} \frac{\psi (s\theta)}{V(s)} = \Lambda (\theta), \quad \theta \in \mathbb{S}^{d-1}.
\end{align}
Let $g$ be a bounded function and $\mu$ a finite Borel measure. Set $r =g\ast \check{\mu}$ and assume that the below limit exists 
\begin{align*}
\lim_{t\to 0^+}t^{-\beta}\left( r(t\theta)+r(-t\theta)-2r(0) \right)=R_\beta(\theta), \quad \mathrm{for\ all}\ \theta \in \mathbb{S}^{d-1}.
\end{align*} 
Moreover, suppose that $r$ satisfies
\begin{align}\label{beta_r_cond}
|r(x)+r(-x)-2r(0)|\leq L\norm{x}^\beta,\quad \mathrm{for}\ L>0.
\end{align}
Then
\begin{align*}
\lim_{t\to 0^+} [V^-(1/t)]^\beta \left( H_g^\mu(t) -H_g^\mu (0) \right) = 
\frac{1}{2}\int_{\RR^d} R_\beta\left( x/ \norm{x} \right) \norm{x}^\beta p_{\Lambda}(x) \ud x ,
\end{align*}
where the density function $p_{\Lambda}(x)$ is uniquely determined by the formula
\begin{align}\label{Lambda_eq}
e^{-\Lambda \left( \frac{x}{\norm{x}} \right)\norm{x}^\alpha} = \int_{\RR^d} e^{i\sprod{x}{y}}p_{\Lambda} (y)\ud y.
\end{align}
\end{theorem}
The particular choice $g= \IndFun{\Omega}$ and $\mu (\ud x)=  \IndFun{\Omega} \ud x$ leads to the result for the classical heat content defined at \eqref{Heat_Content}. Then the function $r (x)= g\ast \check{\mu}(x) = |\Omega \cap (\Omega +x)|$ is the covariance function of the set $\Omega$ and the function $R_\beta$ for $\beta =1$ is determined in terms of the related directional derivative, cf. \cite[Subsection 2.1]{CG} for more details. 

The following corollary gives the asymptotic behaviour when the L\'{e}vy process $\mathbf{X}$ is isotropic and its (radial) L\'{e}vy-Khintchine exponent $\psi (r)$ is a regularly varying function at infinity with index greater than one.
Let us recall that a L\'{e}vy process $\mathbf{X}$ is isotropic if the measure 
$p_t(\ud x)$ is radial (rotationally invariant) for each $t > 0$, which is equivalent to saying that the matrix $A=\lambda I$ for some $\lambda \geq0$, the L\'{e}vy measure $\nu$ is rotationally invariant and $\gamma =0$. For isotropic processes the L\'{e}vy-Khintchine exponent has the specific form
\begin{align*}
\psi (x) = \int_{\RR^d}\left( 1- \cos \sprod{x}{y}\right)\nu (\ud x) + \lambda \norm{x}^2,
\end{align*}
for some $\lambda \geq 0$. 
We usually abuse notation by setting $\psi(r)$ to be equal to $\psi(x)$ for any $x \in \RR^d$
with $\norm{x} = r>0$. 

By $\psi^-$ we denote the generalized inverse of $\psi^*$.
Using \cite[Theorem 1.5.3]{bgt}, if $\psi \in \RInf$, for some $\alpha>0$, then $\psi ^* \in \RInf$ and thus $\psi ^- \in \mathcal{R}_{1/\alpha}$, which implies that $\lim_{t\to 0}\psi^-(1/t) = \infty$.

The precise constant in the below formula is found by an application of a variant of \cite[Eq. (25.6)]{Sato}.

\begin{corollary}\label{Thm_alpha>1}
Let $\beta \in [1,2)$ be fixed. Let $\mathbf{X}$ be an isotropic L\'{e}vy process in $\RR^d$ with
the L\'{e}vy-Khintchine exponent $\psi$ such that $\psi \in \RInf$, $\alpha \in (\beta ,2]$. Let $g$ be a bounded function and $\mu$ a finite Borel measure. Assume that $r =g\ast \check{\mu}$ satisfies the assumptions of Theorem \ref{Multivariate_Reg_Var_thm}.
Then
\begin{align*}
\lim_{t\to 0^+}\, [\psi ^-(1/t)]^\beta \left( H_g^\mu(t) -H_g^\mu (0) \right) = 
\pi^{-d/2}4^{\beta /2-1}\Gamma \left( \frac{d+\beta}{2}\right) \frac{\Gamma \left( 1-\frac{\beta}{\alpha} \right)  }{\Gamma \left( 1-\frac{\beta}{2}\right)}
\int_{\mathbb{S}^{d-1}}R_\beta (\theta)\sigma (\ud \theta) .
\end{align*}
\end{corollary}

The next theorem treats about the assumption on the L\'{e}vy measure, that is we require it is regularly varying according to the presentation by Resnick \cite{Resnick}.
\begin{theorem}\label{Thm_Levy_measure}
Let $\beta \in [1,2)$ be fixed. Let $\mathbf{X}$ be a symmetric L\'{e}vy process in $\RR^d$ with 
the triplet $(0,0 ,\nu )$. Suppose that there is a measure $\eta$ on $\RR^d\setminus\{0\}$ such that
\begin{align}\label{Mult_Reg_Cond_Levy_meas}
\lim_{s\to 0^+} \frac{\nu (sG)}{\nu \left( B_s^{c} \right) } = \eta (G),\quad \mathrm{for\ }
G\subset \RR^d\setminus\{0\}\ \mathrm{with}\ \eta (\partial G)=0,
\end{align}
where $V(t) = \nu (B_{1/t}^c)$ is regularly varying at infinity of index $\alpha \in (\beta ,2)$.
Let $g$ be a bounded function, $\mu$ a finite Borel measure and $r =g\ast \check{\mu}$. Assume that there is a real function $R_\beta$ defined on the sphere $\mathbb{S}^{d-1}$ such that
\begin{align*}
\lim_{t\to 0^+}\sup_{\theta \in \mathbb{S}^{d-1}}\Big\vert \frac{r(t\theta)+r(-t\theta)-2r(0)}{t^\beta}-R_\beta(\theta)\Big\vert = 0.
\end{align*}
Moreover, suppose that $r$ satisfies \eqref{beta_r_cond}.
Then
\begin{align*}
\lim_{t\to 0^+} [V^-(1/t)]^\beta \left( H_g^\mu(t) -H_g^\mu (0) \right) = 
\frac{1}{2}\int_{\RR^d} R_\beta\left( x/\norm{x} \right) \norm{x}^\beta p_{\eta}(x)\ud x ,
\end{align*}
where the density function $p_{\eta}$ is uniquely determined by the formula
\begin{align*}
e^{-\int_{\RR^d} \left( 1-\cos \sprod{\xi}{y}\right) \eta (\ud y)} = \int_{\RR^d} e^{i\sprod{\xi}{x}}p_{\eta}(x)\ud x .
\end{align*}
\end{theorem}
It is worth pointing out that Theorem \ref{Thm_Levy_measure} (as well as Theorem \ref{Multivariate_Reg_Var_thm}) applies in the case when $\mathbf{X}$ is the stable process in $\RR^d$.
We also emphasize that the support of the measure $\eta$ in \eqref{Mult_Reg_Cond_Levy_meas} may be contained in some hyperplane of $\mathbb{R}^d$, see Example \ref{Ex:5}.
Further, condition \eqref{Mult_Reg_Cond_Levy_meas} forces the scaling property of the measure $\eta$, that is there exists some $\alpha \geq 0$ such that $\eta (tG) = t^{-\alpha}\eta (G)$, for all $t>0$ and sets $G$ with $\eta (\partial G)=0$. This in turn implies that $\eta$ in Theorem \ref{Thm_Levy_measure} is the L\'{e}vy measure of the $\alpha$-stable law.

In the rest of the paper we first present a list of examples and concluding Section \ref{sec_Proofs} is devoted to the proofs of the aforementioned results.

\section{Examples}
Let $\Omega$ be a non-empty open subset of $\RR ^d$ such that its Lebesgue measure $|\Omega|$ is finite.  

\begin{example}
Let $\mathbf{X}$ be a L\'{e}vy process with finite variation and $\mathcal{L}^0$ be the generator of the process $X_t^0 =X_t + t\gamma_0$. Let $g(x) = \IndFun{\Omega}(x)$ and $\mu (\ud x) = f(x)\ud x$, where $f\in C_0^1(\RR^d)$ with $\nabla f$ bounded. In particular, $f$ is Lipschitz and $\lim_{\norm{x}\to \infty}f(x)=0$. In this case $r(x) = g\ast \check{f}(x)$ and it is Lipschitz with $\lim_{\norm{x}\to \infty}r(x)=0$, belongs to $\mathrm{Dom}(\mathcal{L}^0)$ and $\nabla r = g\ast \nabla \check{f}$. Moreover, we claim that
$\mathcal{L}^0 r(0) = \int_{\Omega}\mathcal{L}^0 f(x)\ud x$.
Indeed, applying \cite[Lemma 2]{CG} we obtain that
\begin{align*}
\mathcal{L}^0 r(0) &= \int_{\RR^d}\left( r (y) - r (0)\right) \nu (\ud y) = 
\int_{\RR^d}\nu (\ud y)\int_{\RR^d}  f(-x)\left( \IndFun{\Omega}(y-x) - \IndFun{\Omega}(-x)\right) \ud x \\
&= \int_{\RR^d}\nu (\ud y)\int_{\Omega}\left(f(x+y)-f(x)\right)\ud x = 
\int_{\Omega} \int_{\RR^d} \left(f(x+y)-f(x)\right)\nu (\ud y) \ud x \\
&=\int_{\Omega}\mathcal{L}^0 f(x)\ud x.
\end{align*} 
Hence, by Theorem \ref{Thm_X_bdd_variation},
\begin{align*}
\lim_{t\to 0^+}t^{-1}\left(H_g^\mu (t) -H_g^\mu  (0)\right)
 = \int_{\Omega}\mathcal{L}^0 f(x)\ud x
  + \norm{\gamma_0} \nabla_{\frac{\gamma_0}{\norm{\gamma _0}}} r(0) \IndFun{\RR^d\setminus \{0\}}(\gamma _0).
\end{align*}
\end{example}

\begin{example}
Let $\mathbf{X}$ be a L\'{e}vy process in $\RR^d$. Let $g(x) = \IndFun{\Omega}(x)$ and $\mu (\ud x) = \IndFun{\Omega_0}\ud x$, for some $\Omega_0\subset \RR^d$ with $|\Omega_0|<\infty$. We have $r(x) = |\Omega \cap (\Omega_0 +x)|$ and it is bounded, uniformly continuous and vanishes at infinity. We consider two cases:\\
\textit{Case 1.} Let $\Omega \cap \Omega_0 =\emptyset$ with $\mathrm{dist}(\Omega , \Omega_0)=D>0$. Then $r(x) = 0$, for $\norm{x}< D$ and, applying \cite[Corollary 8.9]{Sato}, we obtain 
\begin{align*}
t^{-1}H_g^\mu(t) = t^{-1}\int_{\RR^d} r(x)p_t(\ud x)\longrightarrow \int_{\Omega}\nu \left(y- \Omega_0\right) \ud y ,\quad \mathrm{as}\ t\to 0^+.
\end{align*}
\textit{Case 2.} If $\Omega \subset \Omega _0$ with $\mathrm{dist}(\Omega ,\Omega_0^c)>0$, we similarly get that
\begin{align*}
t^{-1}(H_g^\mu (t) - H_g^\mu (0)) =t^{-1}\int_{\RR^d} (r(x) - r(0))p_t(\ud x)\longrightarrow -\int_{\Omega} \nu (y-\Omega_0^c)\, \ud y,\quad \mathrm{as}\ t\to 0^+. 
\end{align*}
\end{example}

\begin{example}
Let $\mathbf{X}$ be a L\'{e}vy process in $\RR^d$. Let $\mu(\ud x) =f(x)\ud x$ with the function $f(x) = (2 \pi)^{-d/2}e^{-\norm{x}^2/2}$ and suppose that $g\in L^\infty (\RR^d)$. Then $\check{f}=f$ and $g\ast f \in C_0^\infty (\RR^d)$, and whence it also belongs to $\mathrm{Dom}(\mathcal{L})$. Since $\nabla f$ is bounded, we deduce that $f$ is Lipschitz. By \eqref{Generator_formula} we obtain that
\begin{align*}
\lim_{t\to 0^+} t^{-1} \left( H_g^\mu (t) -\int_{\RR^d} f(x)g(x)\ud x \right) = 
\mathcal{L}(g\ast f)(0) .
\end{align*}
Similarly we can apply Theorem \ref{Thm_X_bdd_variation}.
\end{example}

\begin{example}
Let $S^{(\alpha)}$ be the $\alpha$-stable process in $\RR^d$ with $0<\alpha <2$ and with the L\'{e}vy measure $\nu ^{(\alpha )}$ given by the formula
\begin{align*}
\nu ^{(\alpha)}(B) = \int_{\mathbb{S}^{d-1}}\!\!\! m^{(\alpha)}(\ud \theta ) \int_0^\infty \IndFun{B}(r\theta )\frac{\ud r}{r^{1+\alpha}},\quad \mathrm{for}\ B\in \mathcal{B}(\RR),
\end{align*}
where $m^{(\alpha)}$ is a finite measure on the sphere $\mathbb{S}^{d-1}$, cf. \cite[Theorem 14.3]{Sato}. 
We abuse notation and we write just $m$ for the measure $m^{(\alpha)}$.
We additionally assume that there is no hyperplane $\mathcal{V}$ of $\mathbb{R}^d$
such that $m$ is supported in $\mathcal{V}$. The corresponding L\'{e}vy-Khintchine exponent $\psi^{(\alpha)}$ takes the form
\begin{align*}
\psi^{(\alpha)} (x) = \int_{\mathbb{S}^{d-1}} \int_0^\infty \left( 1- \cos \sprod{x}{r\theta} \right) \frac{\ud r}{r^{1+\alpha}} m(\ud \theta ).
\end{align*}
Consider a symmetric L\'{e}vy process $\mathbf{X}$ of which the L\'{e}vy-Khintchine exponent $\psi$ equals
\begin{align*}
\psi (x) = \int_{\mathbb{S}^{d-1}} \int_0^\infty \left( 1- \cos \sprod{x}{r\theta} \right) \frac{f(1/r)}{r}\, \ud r\,  m(\ud \theta ),
\end{align*}
for a given function $f\in \RInf$. The corresponding L\'{e}vy measure is 
\begin{align*}
\nu_{f} (B) = \int_{\mathbb{S}^{d-1}}\!\!\! m(\ud \theta ) \int_0^\infty \IndFun{B}(r\theta )\frac{f(1/r) }{r}\, \ud r,\quad \mathrm{for}\ B\in \mathcal{B}(\RR).
\end{align*}
It follows, cf. \cite[Remark 14.4]{Sato}, that for any non-negative and measurable function $F$ we have
\begin{align}
\int_{\RR^d} F(x)\nu_f (\ud x) = \int_{\mathbb{S}^{d-1}} \int_0^\infty F(r\theta) \frac{f(1/r) }{r}\, \ud r\,  m(\ud \theta ) .\label{FormulaHelp}
\end{align} 

Our aim is to apply Theorem \ref{Multivariate_Reg_Var_thm} to the process $\mathbf{X}$. For that reason we need to verify condition \eqref{PsiStar} for the function $\psi$.
We first show that there are some $R>0$ and positive constants $c=c(f),C=C(f)$ such that for all $s\geq R $ we have
\begin{align}
cf(s)\leq \psi (s\theta_0)\leq Cf(s),\quad \mathrm{for\ each}\ \theta_0\in \mathbb{S}^{d-1}.\label{eq_*}
\end{align}
Indeed, for the upper bound, by a suitable change of variable we write
\begin{align*}
\frac{\psi(s\theta_0)}{f(s)} =  \int_{\mathbb{S}^{d-1}} \int_0^\infty \left( 1- \cos \big( \rho \sprod{\theta_0}{\theta}\big) \right) \frac{f(s/\rho)}{f(s)}\frac{\ud \rho}{\rho}\,  m(\ud \theta ).
\end{align*}
By \eqref{eq:14}, for any $\varepsilon >0$ there is $R>0$ such that
\begin{align*}
\frac{f(s/\rho)}{f(s)} \leq 2 \max \{ (1/\rho)^{\alpha +\varepsilon}, (1/\rho)^{\alpha -\varepsilon} \}, \qquad s\geq \rho R,\ s\geq R.
\end{align*}
Thus we get for $s\geq R$
\begin{multline*}
\frac{\psi(s\theta_0)}{f(s)} \leq 2\int_{\mathbb{S}^{d-1}} \int_0^1 \left( 1- \cos \big( \rho \sprod{\theta_0}{\theta}\big) \right) \frac{\ud \rho}{\rho^{\alpha +\varepsilon +1}}\,  m(\ud \theta ) \\
+
2\int_{\mathbb{S}^{d-1}}\int_1^{s/R} \left( 1- \cos \big( \rho \sprod{\theta_0}{\theta}\big) \right) \frac{\ud \rho}{\rho^{\alpha -\varepsilon +1}}\,  m(\ud \theta ) \\
+ \int_{\mathbb{S}^{d-1}} \int_{s/R}^\infty \left( 1- \cos \big( \rho \sprod{\theta_0}{\theta}\big) \right) \frac{f(s/\rho)}{f(s)}\frac{\ud \rho}{\rho}\,  m(\ud \theta ).
\end{multline*}
Since $1- \cos \big( \rho \sprod{\theta_0}{\theta}\big)\leq 2\min \{ 1, \rho^2 \}$, we can estimate the two first integrals by the quantity $Km(\mathbb{S}^{d-1})$, for some $K>0$. 
To the last integral we use formula \eqref{FormulaHelp} which yields
\begin{align*}
\int_{\mathbb{S}^{d-1}} \int_{s/R}^\infty \left( 1- \cos \big( \rho \sprod{\theta_0}{\theta}\big) \right) &\frac{f(s/\rho)}{f(s)}\frac{\ud \rho}{\rho}\,  m(\ud \theta )\\
&= \frac{1}{f(s)}\int_{\norm{x}> \frac{1}{R}} \left( 1- \cos \sprod{\theta_0}{x} \right) \nu_f(\ud x)
\leq \frac{2}{f(s)}\nu_f(B_{1/R}^c).
\end{align*}
As $f$ diverges to infinity, the upper bound independent of $\theta_0$ is found. For the lower bound we use again Potter bounds \eqref{eq:14}. For $\varepsilon >0$ there is $R>0$ such that
\begin{align*}
\frac{f(s/\rho)}{f(s)} \geq \frac{1}{2} \min \{ (1/\rho)^{\alpha +\varepsilon}, (1/\rho)^{\alpha -\varepsilon} \}, \qquad s\geq \rho R,\ s\geq R.
\end{align*}
Taking $\varepsilon = \alpha /2$ we obtain
 \begin{align*}
\frac{\psi(s\theta_0)}{f(s)} \geq  \frac{1}{2} \int_{\mathbb{S}^{d-1}} \int_0^1 \left( 1- \cos \big( \rho \sprod{\theta_0}{\theta}\big) \right) \frac{\ud \rho}{\rho^{1+ \frac{\alpha}{2}}}\,  m(\ud \theta ) = \frac{1}{2} \widetilde{\psi}^{(\alpha /2)}(\theta_0),
\end{align*}
where $\widetilde{\psi}^{(\alpha /2)}$ is the L\'{e}vy-Khintchine exponent of the pure-jump truncated $(\alpha /2)$-stable process with the L\'{e}vy measure $\nu ^{(\alpha /2)}|_{B_1}$. In particular, the function $\widetilde{\psi}^{(\alpha /2)}$ is continuous. We finally observe that $\inf _{\theta \in \mathbb{S}^{d-1}}\widetilde{\psi}^{(\alpha /2)}(\theta)>0$. Indeed, suppose \textit{a contrario} that this infimum is zero. Then there must exist $\theta_1\in \mathbb{S}^{d-1}$ such that $\widetilde{\psi}^{(\alpha /2)}(\theta_1)=0$. But then the support of the measure $m$ is contained in the subspace $\{ \lambda \theta_1 :\, \lambda \in \mathbb{R}\}^{\bot}$. This contradicts our assumption and hence we have also found the lower bound in \eqref{eq_*}.

Since $f\in \RInf$ with $\alpha >0$, we have $\sup_{R\leq r\leq s}f(r)\asymp f(s)$, cf. \cite[Theorem 1.5.3]{bgt}.
With the aid of \eqref{eq_*} we conclude that $\sup_{\theta \in \mathbb{S}^{d-1} ,\, R\leq r\leq s}\psi(r\theta_0)\asymp f(s)$ and thus condition \eqref{PsiStar} is satisfied for all $\norm{x}\geq R$. But clearly continuity of $\psi$ allows us to deduce the result also for $1\leq \norm{x}\leq R$ and the proof of \eqref{PsiStar} is finished.

We next observe that 
\begin{align*}
\lim_{s\to \infty}\frac{\psi (s\theta)}{f(s)} = \psi^{(\alpha)}(\theta),\quad \theta \in \mathbb{S}^{d-1}.
\end{align*}
This follows by an application of the Lebesgue dominated convergence theorem which is justified by an analogous argument to that one used for the upper bound in \eqref{eq_*}.

Finally, let $\mu $ be a finite measure and $g$ be a bounded function. Assume that the function $r = g\ast \check{\mu}$ satisfies all the conditions of Theorem \ref{Multivariate_Reg_Var_thm}. We obtain
\begin{align*}
\lim_{t\to 0^+} [f^-(1/t)]^\beta \left( H_g^\mu(t) -H_g^\mu (0) \right) = 
\frac{1}{2}\int_{\RR^d} R_\beta\left( x/ \norm{x} \right) \norm{x}^\beta p^{(\alpha)}(x) \ud x ,
\end{align*}
where the density function $p^{(\alpha)}(x)$ is uniquely determined by
$e^{-\psi^{(\alpha)} \left(x \right)} = \int_{\RR^d} e^{i\sprod{x}{y}}p^{(\alpha)} (y)\ud y$.
\end{example}

\begin{example}\label{Ex:5}
Let $\mathbf{X} = (S^{(\alpha )}, S^{(\rho )})$, where $S^{(\alpha )}$ and $S^{(\rho )}$ are two independent symmetric stable processes in $\RR$ with indexes $\alpha$ and $\rho$ respectively and such that $0<\alpha <\rho <2$. The L\'{e}vy measure $\nu$ of $\mathbf{X}$ is supported on axes $OX$ and $OY$. Condition \eqref{Mult_Reg_Cond_Levy_meas} forces that the same holds for the limit measure $\eta$. Moreover, since in this case $V(t) = \nu (B_{1/t}^c)$ belongs to $\mathcal{R}_\rho$, we conclude that $\eta (OX) =0$ and thus $\eta$ is the symmetric $\rho$-stable measure supported on $OY$. 

Let $\beta =1 $, $1<\rho<2$ and set $g=\IndFun{\Omega}$, $\ud \mu  =\IndFun{\Omega}\ud x $, for a radial set $\Omega$. We shall apply Theorem \ref{Thm_Levy_measure}.
With this choice we have that $r(x) = g\ast \check{\mu }(x) = |\Omega \cap (\Omega +x)|$ is the so-called covariance function of the set $\Omega$ for which $\lim_{t\to 0^+}t^{-1}(r(0)-r(t\theta))=V_\theta (\Omega)/2$, where $V_\theta(\Omega)$ is the directional derivative of $\IndFun{\Omega}$ in the direction $\theta \in \mathbb{S}^1$.
For sets of finite perimeter the following relation holds
\begin{align}\label{Per}
\mathrm{Per}(\Omega) = \frac{\Gamma \left( \frac{d+1}{2}\right)}{\pi ^{(d-1)/2}}\int_{\mathbb{S}^{d-1}}V_\theta (\Omega)\sigma (\ud \theta),\quad \Omega \subset \RR^d.
\end{align}
For all of this we refer the reader to \cite[Subsection 2.1]{CG}. In particular, see eg. \cite[Eq. (4)]{CG} for the precise definition of the perimeter $\mathrm{Per}(\Omega)$ of the set $\Omega$, cf. also \cite{Ambrosio_2000} and\cite{Galerne}.

By our choice of the function $g$ and the measure $\mu$, we have $H_g^\mu (0) - H_g^\mu (t) = |\Omega|-H_\Omega (t)$, where $H_\Omega (t)$ is the heat content defined at \eqref{Heat_Content}.
Moreover, since for radial sets $V_\theta (\Omega)$ is constant, setting $e_2=(0,1)$ we obtain that 
\begin{align*}
\lim_{t\to 0^+}V^{-}(1/t) \left( |\Omega| - H_{\Omega} (t)\right)  = \frac{V_{e_{2}}(\Omega)}{2}\int_0^\infty xp_\eta ^{(\rho)} (x)\ud x = \pi ^{-2}\Gamma \left( 1-\frac{1}{\rho}\right) \mathrm{Per}(\Omega),
\end{align*}
where for the last equality we used \eqref{Per} together with the formula for the expectation of the stable random variable, cf. \cite[Eq. (25.6)]{Sato}.
We mention that for non-radial sets the last equality in the above formula is not valid. Indeed, if we take $\Omega$ to be the rectangle centered at $(0,0)$ and with sides of length $0<a<b$, then one easily computes that $V_{e_2}(\Omega) =V_{-e_2}(\Omega)=4a$ and whence 
\begin{align*}
\lim_{t\to 0^+}V^{-}(1/t)\left( |\Omega| - H_{\Omega} (t)\right) = \frac{V_{e_{2}}(\Omega)+V_{-e_{2}}(\Omega)}{4}\int_0^\infty xp_\eta ^{(\rho)}(x)\ud x = 2\pi^{-1} \Gamma \left( 1-\frac{1}{\rho}\right) a .
\end{align*}
\end{example}

\section{Proofs}\label{sec_Proofs}
We start with an auxiliary lemma which is closely related to the small-time moment behaviour of L\'{e}vy processes studied in \cite{Jacod}, \cite{Lopez} and \cite{Kuhn}. In particular, we extend an admissible class of functions from \cite[Section 5.2]{Jacod} in the case $\beta =1$, and the result \cite[Theorem 3.5]{Kuhn} for L\'{e}vy processes. 

After \cite{Pruitt} we consider the following function related to the L\'{e}vy process $\mathbf{X}$, for any $r>0$,
\begin{align}
\begin{aligned}\label{Pruitt_Function}
h(r)= \norm{A}r^{-2} &+ 
 r^{-1} \Big\lVert \gamma +\int_{\RR ^d} \left(\textbf{1}_{\norm{y} < r}-\textbf{1}_{\norm{y} < 1}\right) y\, \nu(\ud y)\Big\rVert  \\
& +\int_{\RR ^d} \left( 1\wedge \norm{y}^2r^{-2}\right) \, \nu(\ud y),
\end{aligned}
\end{align}
where $(A,\gamma ,\nu)$ is the triplet from \eqref{charact_expo} and $\Vert A\Vert = \max_{\Vert x\Vert =1} \norm{Ax}$.
We shall repeatedly use the estimate \cite[Formula (3.2)]{Pruitt}; there is some positive constant $C=C(d)$ such that for any $r>0$,
\begin{align}\label{estimate_Pruitt}
\PP \left( \norm{X_t}\geq r \right) \leq \PP \left(\sup_{0\leq s\leq t } \norm{X_s}\geq r \right)\leq C t h(r) .
\end{align}
We also recall that for symmetric L\'{e}vy processes, see \cite[Corollary 1]{Grzywny1},
\begin{equation}\label{psi_star_estimate}
	\frac{1}{2} \psi^*(r^{-1}) \leq h(r) \leq 8(1 + 2d) \psi^*(r^{-1}).
\end{equation}

\begin{lemma}\label{Lemma_est}
Let $\mathbf{X}$ be a L\'{e}vy process in $\RR^d$ with the triplet $(0, \gamma,  \nu)$ and such that
\begin{align*}
\int_{\norm{x}<1} \norm{x}^\beta \nu (\ud x) <\infty ,\quad\ 0\leq \beta \leq 2.
\end{align*}
For $\beta \in [0,1]$ we additionally assume that $\gamma_0=0$.
Let $F\in C_b(\RR^d)$ satisfy $|F(x)-F(0)|\leq C\norm{x}^\beta$, for $\norm{x}<1$. Then
\begin{align*}
\lim_{t\to 0^+}t^{-1}\int_{\RR^d} \left( F(x)-F(0)\right) p_t(\ud x) = \int_{\RR^d} \left( F(x)-F(0)\right)\nu (\ud x).
\end{align*}
\end{lemma}
\begin{proof}
We choose $0<\varepsilon <1$ and a function $\chi_\varepsilon \in C_c^\infty (\RR^d)$ such that $0\leq \chi_\varepsilon \leq 1$, $\chi_\varepsilon (x)=1$ for $\norm{x}<\varepsilon /2$, and $\chi_\varepsilon (x)=0$ for $\norm{x}>\varepsilon$. We write
\begin{align*}
\int_{\RR^d} \left( F(x)-F(0)\right) p_t(\ud x) &= \int_{\RR^d} \left( F(x)-F(0)\right)\chi_\varepsilon (x) p_t(\ud x)\\ 
&\qquad +
\int_{\RR^d} \left( F(x)-F(0)\right)\left( 1-\chi_\varepsilon (x)\right) p_t(\ud x) = I_\varepsilon (t)+ II_\varepsilon (t).
\end{align*}
By \cite[Corollary 8.9]{Sato}, 
\begin{align*}
\lim_{t\to 0^+}t^{-1}II_\varepsilon (t) = \int_{\RR^d}  \left( F(x)-F(0)\right)\left( 1-\chi_\varepsilon (x)\right) \nu (\ud x).
\end{align*}
Using our assumption we estimate the first integral as follows
\begin{align*}
|I_\varepsilon (t)|\leq C \int_{\RR^d} \phi_\beta (x) p_t(\ud x),\quad \mathrm{where}\ \phi_\beta = \chi_\varepsilon (x)\norm{x}^\beta . 
\end{align*}
We observe that for $\beta >1$ the gradient $\nabla \phi_\beta (0) = 0$ and thus applying \cite[Theorem 4.1]{Kuhn} we obtain that
\begin{align*}
t^{-1}|I_\varepsilon (t)|\leq C\, t^{-1}\! \int_{\RR^d} \phi_\beta (x) p_t(\ud x) \xrightarrow{t\to 0^+}  C \int_{\RR^d} \phi_\beta (x) \nu (\ud x)\leq C_1 \varepsilon.
\end{align*}
The proof is finished.
\end{proof}

\begin{proof}[Proof of Theorem \ref{Thm_X_bdd_variation}]
We start with the case $0\leq \beta \leq 1$. Applying formula \eqref{Generator_formula} and Lemma \ref{Lemma_est} we obtain that
\begin{align*}
	\lim_{t\to 0^+}t^{-1}\left( H_g^\mu (t) - H_g^\mu (0)\right)
	 = \int_{\RR^d} (r(x)-r(0))\nu (\ud x).
	\end{align*}
	
Suppose that $1\leq \beta <2$ and $\mathbf{X}$ is symmetric. We set $F(x) = r(x)+r(-x)$ and then $|F(x)-F(0)|\leq C\norm{x}^\beta$, for all $\norm{x}<1$. We also observe that by symmetry $\gamma_0=0$ when $\beta =1$. Thus, by Lemma \ref{Lemma_est} we conclude that 
\begin{align*}
\lim_{t\to 0^+} t^{-1}\int_{\RR^d} (F(x)-F(0))p_t(\ud x) = \int_{\RR^d} (F(x)-F(0)) \nu (\ud x),
\end{align*}
and symmetry implies $t^{-1}\int_{\RR^d} (F(x)-F(0))p_t(\ud x) = 2\, t^{-1}( H_g^\mu (t) - r (0) )$, which gives the result.

Next, we consider the case when $\mathbf{X}$ is a general L\'{e}vy process. We set 
$$F(x) = r(x) - \sprod{x}{\nabla r(0)}\chi(x),$$ 
where $\chi$ is a compactly supported smooth function such that $0\leq \chi \leq 1$ and it is one for $\norm{x}\leq 1$ and zero for $\norm{x}>2$. Then our assumption implies that $|F(x)-F(0)|\leq C\norm{x}^\beta$, for $\norm{x}<1$. Thus for $1<\beta <2$, by Lemma \ref{Lemma_est}, we conclude that 
\begin{align}
\lim_{t\to 0^+} t^{-1}\int_{\RR^d} (F(x)-F(0))p_t(\ud x) &= \int_{\RR^d} (F(x)-F(0)) \nu (\ud x)\label{Levy_Lim}\\
&= \int_{\RR^d} \left( r(x)-r(0)-\sprod{x}{\nabla r(0)}\chi(x)\right) \nu (\ud x).\notag
\end{align}
To get the same limit in the case when $\beta =1$ we proceed as in Lemma \ref{Lemma_est}. For $0<\varepsilon <1$ we pick a function $\chi_\varepsilon \in C_c^\infty (\RR^d)$ such that $0\leq \chi_\varepsilon \leq 1$, $\chi_\varepsilon (x)=1$ for $\norm{x}<\varepsilon /2$, and $\chi_\varepsilon (x)=0$ for $\norm{x}>\varepsilon$. Then we have 
\begin{align*}
\lim_{t\to 0^+} t^{-1}\int_{\RR^d} \left( F(x)-F(0)\right)(1-\chi_\varepsilon (x))p_t( \ud x)  =
 \int_{\RR^d} \left( F(x)-F(0)\right)(1-\chi_\varepsilon (x)) \nu (\ud x).
\end{align*} 
We observe that $\nabla F(0)=0$ and whence $F(x)-F(0)= o(\norm{x})$, which allows us to estimate the remaining integral as follows
\begin{align*}
\frac{1}{t}\left\vert \int_{\RR^d} \left( F(x)-F(0)\right)\chi_\varepsilon (x) p_t( \ud x)\right\vert &\leq \frac{\varepsilon}{t}\int_{\RR^d}\norm{x}\chi_{\varepsilon}(x)p_t(\ud x) \\
&= \frac{\varepsilon}{t}\int_{\RR^d }\norm{x- t \gamma _0 }\chi_{\varepsilon}(x- t \gamma_0)p_t^0(\ud x)\\
&\leq \frac{\varepsilon}{t}\int_{\RR^d}\norm{x}\chi_{2\varepsilon}(x)p_t^0(\ud x) +\varepsilon \norm{\gamma _0},
\end{align*} 
where $p_t^0(\ud x)$ is the transition probability of the process $\mathbf{X}^0$ which is shifted by $\gamma_0$, i.e. $X_t^0 = X_t+t\gamma_0$. Since the function $x\mapsto \norm{x}\chi_{2\varepsilon} (x)$ is Lipschitz, we apply \cite[Theorem 4.1]{Kuhn} and deduce that the last integral tends to $\int_{\RR^d}\norm{x}\chi_{2\varepsilon}(x)\nu (\ud x)$, which implies \eqref{Levy_Lim}.
Finally, we write 
\begin{align*}
t^{-1}\int_{\RR^d} (F(x)-F(0))p_t(\ud x) =  t^{-1}( H_g^\mu (t) - r (0) ) -  \Big\langle t^{-1} \int_{\RR^d} x \chi(x) p_t(\ud x),\, \nabla r(0) \Big\rangle .
\end{align*}
The function $x\chi (x)\in C_{0}^{\infty}(\RR^d)$ and thus \eqref{Gener_form} yields
\begin{align*}
\lim_{t\to 0^+} \Big\langle t^{-1} \int_{\RR^d} x \chi(x) p_t(\ud x),\, \nabla r(0) \Big\rangle
= 
\Big\langle \gamma + \int_{\RR^d} x\left( \chi(x) - \IndFun{\{\norm{x}\leq 1\}}\right) \nu (\ud x),\, \nabla r(0) \Big\rangle.
\end{align*}
Hence
\begin{align*}
\lim_{t\to 0^+} t^{-1}( H_g^\mu (t) - H_g^\mu (0) ) = 
\sprod{\gamma}{\nabla r(0)} + \int_{\RR^d} \left( r(x)-r(0)-\sprod{x}{\nabla r(0)}\IndFun{\{\norm{x}\leq 1\}} \right) \nu (\ud x)
\end{align*}
and the proof is finished.
\end{proof}

Before we prove Theorem \ref{Multivariate_Reg_Var_thm} we state an auxiliary lemma.
\begin{lemma}\label{Lemma_stable}
Let $\mathbf{X}$ be a symmetric L\'{e}vy process in $\RR^d$ with the transition probability $p_t(\ud x)$. Assume that its L\'{e}vy-Khintchine exponent $\psi $ satisfies \eqref{Mult_Reg_Cond} with functions $V \in \RInf$, $\alpha \in (0,2]$ and continuous $\Lambda \colon \mathbb{S}^{d-1}\to (0,\infty)$, and that condition \eqref{PsiStar} holds. Then $p_t(\ud x)=p_t(x)\ud x$ and
\begin{align}\label{limit_stable}
\lim_{t\to 0^+} \frac{p_t\left( \frac{x}{\psiI} \right)}{(\psiI )^d} = p_\Lambda (x),
\end{align}
where $p_\Lambda$ is the density defined at \eqref{Lambda_eq}. 
\end{lemma}

\begin{proof}
Conditions \eqref{PsiStar} and \eqref{Mult_Reg_Cond} imply that
\begin{align*}
\lim_{x\to \infty}\frac{\psi (x)}{\log (1+\norm{x})}=\infty,
\end{align*}
and whence $p_t(\ud x) = p_t(x)\ud x$ with the density $p_t\in L_1(\RR^d)\cap C_0(\RR^d)$, see e.g. \cite[Theorem 1]{Knopova_Schilling}.
By the Fourier inversion formula, see \cite[Section 3.3]{Appl}, 
\begin{align}\label{integral_1}
\frac{p_t\left( \frac{x}{\psiI} \right)}{(\psiI )^d} =
\frac{1}{(2\pi)^d} \int_{\RR ^d}\cos \sprod{x}{\xi} e^{-t\psi \left( \psiI \xi\right)}\ud \xi .
\end{align}
By \cite[Theorem 1.5.12]{bgt}, $tV(\psiI)\to 1$. Set $\theta =\xi /\norm{\xi}$ and then we get that
\begin{align*}
\frac{\psi \left( \psiI \xi\right)}{1/t}&= \frac{\psi \left( \psiI \xi\right)}{V \left( \psiI \right)}\cdot
\frac{V(\psiI)}{1/t}\\
&\sim \frac{\psi \left( \psiI \norm{\xi} \theta \right)}{V\left( \psiI \norm{\xi} \right) }
\cdot
\frac{V\left( \psiI \norm{\xi} \right) }{V \left( \psiI \right)}
\to \Lambda (\theta)\norm{\xi}^{\alpha},
\quad t\to 0^+, 
\end{align*}
and this leads to
\begin{align*}
\lim_{t\to 0^+}e^{-t\psi \left( \psiI \xi\right)} = e^{-\Lambda (\theta) \norm{\xi}^\alpha}.
\end{align*}
Therefore, to finish the proof we apply the Dominated convergence theorem. 
First observe that equation \eqref{PsiStar} followed by \eqref{Mult_Reg_Cond} implies that
\begin{align}
\psi^*(r)\asymp V(r),\quad r\geq 1.\label{psi_star}
\end{align}
Now we split the integral in \eqref{integral_1} into two parts. 
According to Potter bounds \eqref{eq:14} applied for the function $V$, there is $r_0>0$ such that, for $t$ small enough and $\norm{\xi}\geq r_0$, 
\begin{align*}
t\psi \left( \psiI \xi\right)\geq \frac{1}{2}
\frac{\psi \left( \psiI \norm{\xi} \theta \right)}{V\left( \psiI \norm{\xi} \right) }
\cdot
\frac{V\left( \psiI \norm{\xi} \right) }{V \left( \psiI \right)}
\geq C \norm{\xi}^{\alpha/2},
\end{align*}
for some $C>0$ which does not depend on $\xi$. Here we used the fact that $\Lambda$ is bounded from below and \eqref{PsiStar} followed by \eqref{psi_star}.
This implies that $e^{-t\psi \left( \psiI \xi\right)}\leq e^{-C\norm{\xi}^{\alpha/2}}$, for $\norm{\xi}\geq r_0$ and $t$ small enough.
For $\norm{\xi}< r_0$ we bound $e^{-t\psi \left( \psiI \xi\right)}$ by one. 
The Dominated convergence theorem followed by the Fourier inversion formula proves \eqref{limit_stable}.
\end{proof}

\begin{proof}[Proof of Theorem \ref{Multivariate_Reg_Var_thm}.]
The proof is based on that of \cite[Theorem 2]{CG} but it requires numerous adjustments and improvements. 
We split the integral in \eqref{Generator_formula} into two parts
\begin{align*}
 H_g^\mu (t) - H_g^\mu (0) &=  \int_{\norm{x}\leq \frac{M}{\psiI}} p_t(x)\left( r (x) - r (0)\right) \ud x \\
 & \quad + 
  \int_{\norm{x}> \frac{M}{\psiI}} p_t(x) \left( r (x) - r (0)\right) \ud x 
  = I_1(t)+I_2(t),\nonumber
\end{align*}
for some fixed $M>1$.
We estimate $I_2(t)$ as follows
\begin{align}\label{INt}
\begin{aligned}
 \Big\vert\int_{\norm{x}> \frac{M}{\psiI}} p_t(x) \left( r (x) - r (0)\right) \ud x \Big\vert &\leq 
 C \int_{\norm{x}> \frac{M}{\psiI}} \left( 1\wedge \norm{x}^\beta \right) p_t(\ud x) \\
 &= C \int_{\norm{x}> \frac{M}{\psiI}}\int_0^{1\wedge \norm{x}^\beta }\ud u\, p_t(\ud x) \\
 &=C \int_0^1 \PP \left( \norm{X_t}> \frac{M}{\psiI}\vee u^{1/\beta} \right)\ud u .
 \end{aligned}
\end{align}
This yields that
\begin{align*}
\left( \psiI \right)^\beta |I_2| \leq C M^{\beta}\PP \Big( \norm{X_t}>\frac{M}{\psiI}\Big) 
	\!+\! C \left( \psiI \right)^\beta \!\! \int_{\! \left( M/ \psiI\right)^{\beta }} ^1 \!\!\!\!\!\!\! \PP \left(\norm{X_t}>u^{1/\beta}\right) \ud u. 
\end{align*}
Thus using \eqref{estimate_Pruitt} followed by \eqref{psi_star_estimate},
\eqref{psi_star} and Potter bounds \eqref{eq:14} for $V$, we get that for $t$ small enough and for $0<\varepsilon < \alpha -\beta$,
\begin{align*}
M^\beta \PP \left( \norm{X_t}> M/\psiI \right) &\leq M^\beta t\psi^*\left( \psiI /M \right)\\ 
&\leq C_1 M^\beta \frac{V\left( \psiI /M \right)}{V\left( \psiI \right)}
\leq C_2 M^{\beta-\alpha +\varepsilon}.
\end{align*}
We proceed similarly with the second term. Applying Karamata's theorem \cite[Proposition 1.5.8]{bgt} and Potter bounds we obtain that for $t$ small enough
\begin{align*}
\left( \psiI \right)^\beta \int_{(M/\psiI )^\beta}^1 \PP \left( \norm{X_t}> u^{1/\beta } \right)\ud u 
&\leq C t\left( \psiI \right)^\beta   \int_{M/\psiI}^1 V(u^{-1})u^{\beta -1}\, \ud u \\ 
&\leq C_1 \frac{M^\beta}{\alpha -\beta}\, t\, V\left( \psiI /M \right)\\
&\leq C_2  \frac{M^\beta}{\alpha -\beta} \frac{V\left( \psiI /M \right)}{V\left( \psiI \right)}
	\leq C_3  \frac{M^{\beta-\alpha +\varepsilon}}{\alpha -\beta}.
\end{align*}

We are left to study the integral $I_1(t)$: by a change of variable we get 
\begin{align*}
[\psiI ]^\beta I_1(t) =\frac{1}{2} \int_{\norm{x}<M}
 \frac{p_t\left(\frac{x}{\psiI}\right)}{(\psiI)^d} K_\beta (x,t)\norm{x}^\beta \ud x,
\end{align*}
where 
\begin{align*}
K_\beta (x,t) = \frac{\left( r\left(\frac{x}{\psiI}\right)+r\left(-\frac{x}{\psiI}\right) - 2r (0)\right)}{\norm{x}^\beta/(\psiI)^\beta}.
\end{align*}
We claim that for any fixed $M>0$,
\begin{align}\label{claim111}
\lim_{t\to 0^+} \int_{\norm{x}<M}\!\! \frac{p_t\left( \frac{x}{\psiI}\right) }{(\psiI )^d}\,K_\beta (x,t)  \norm{x}^\beta \ud x
=
\int_{\norm{x}<M}  R_\beta \left( \frac{x}{\norm{x}}\right)  \norm{x}^\beta p_{\Lambda}(x) \ud x,
\end{align}
where $p_\Lambda (x)$ is given by \eqref{Lambda_eq}.
To show the claim we use the Dominated convergence theorem. By \eqref{beta_r_cond},
\begin{align*}
|K_\beta (x,t) |\leq L \quad
\textrm{and}\quad  
\lim_{t\to 0^+}K_\beta (x,t) = R_\beta \left( \frac{x}{\norm{x}}\right), \textrm{\ for\ any\ } x .
\end{align*}
Next, by \cite[Formula (23)]{BGR}, for $t$ small enough,					 
\begin{align*}
\frac{p_t\left( \frac{x}{\psiI} \right)}{(\psiI )^d} \leq \frac{p_t(0)}{(\psiI )^d}
\leq C,
\end{align*}
and, by Lemma \ref{Lemma_stable},				
\begin{align*}
\lim_{t\to 0^+} \frac{p_t\left( \frac{x}{\psiI}\right)}{(\psiI )^d} = p_\Lambda (x).
\end{align*}
The Dominated convergence theorem implies \eqref{claim111}.

If we let $M$ tend to infinity we finally conclude that.
\begin{align*}
\lim_{t\to 0^+}[\psiI]^\beta I_1 (t)= \frac{1}{2}
\int_{\RR^d} R_\beta \left( \frac{x}{\norm{x}}\right) \norm{x}^\beta p_\Lambda (x) \ud x
\end{align*}
and the result follows.
\end{proof}

\begin{proof}[Proof of Theorem \ref{Thm_Levy_measure}]
Take a smooth function $\chi_{\varepsilon}$ such that $0\leq \chi_{\varepsilon} \leq 1$ and it is one for $\varepsilon <\norm{x}<1/\varepsilon$, and zero for $2/\varepsilon >\norm{x}$ or $\norm{x}<\varepsilon /2$. 
By \eqref{Generator_formula} we can write
\begin{align*}
2\left( H_g^\mu (t) - H_g^\mu (0)\right) &=  \int_{\RR^d}  \left( r (x) +r(-x) - 2r (0)\right)  \left( 1- \chi_{\varepsilon}\left(\psiI x\right) \right) p_t(\ud x)
\\ &\quad + 
  \int_{\RR^d}  \left( r (x) +r(-x) - 2r (0)\right) \chi_{\varepsilon}\left( \psiI x\right)  p_t(\ud x) .
\end{align*}
We have
\begin{align*}
\left\vert \int_{\RR^d}  \left( r (x) +r(-x) - 2r (0)\right)  \left( 1- \chi_{\varepsilon}\left(\psiI x\right) \right) p_t(\ud x)\right\vert
&\leq C_1\left( \frac{\varepsilon}{\psiI}\right)^\beta \\
& + C_2\int_{\norm{x}>1/(\varepsilon \psiI)} \!\!\!\!\!\!\!\!\!\!\!\!\!\!\! \left( 1\wedge \norm{x}\right) ^\beta p_t(\ud x).
\end{align*}
We show that the last integral is small similarly as in the proof of Theorem \ref{Multivariate_Reg_Var_thm}, cf. \eqref{INt}. This and a change of variable yield that
\begin{align*}
2[\psiI]^\beta \left( H_g^\mu (t) - r (0)\right)  = o(1)+ 
\int_{\RR^d}  K_\beta (x,t) \norm{x}^\beta \chi_{\varepsilon}(x) \widetilde{p}_t(\ud x) ,
\end{align*}
where $\widetilde{p}_t(G) = p_t\left(G / \psiI \right)$, for any Borel set $G\subset \RR^d$ and 
\begin{align*}
K_\beta (x,t) = \frac{ r\left(\frac{x}{\psiI}\right)+r\left(-\frac{x}{\psiI}\right) - 2r (0) }{\norm{x}^\beta/(\psiI)^\beta}.
\end{align*}
Further we write
\begin{align*}
\int_{\RR^d}  K_\beta (x,t) \norm{x}^\beta \chi_{\varepsilon}(x) \widetilde{p}_t(\ud x)
&=
\int_{\RR^d} \left( K_\beta (x,t) -R_\beta\left( \frac{x}{\norm{x}}\right)\right)  \norm{x}^\beta \chi_{\varepsilon}(x) \widetilde{p}_t(\ud x)\\
& \qquad +
\int_{\RR^d}  R_\beta\left( \frac{x}{\norm{x}}\right) \norm{x}^\beta \chi_{\varepsilon}(x) \widetilde{p}_t(\ud x)
\end{align*}
and the first integral we estimate as follows
\begin{align*}
\Big\vert \int_{\RR^d} \left( K_\beta (x,t) -R_\beta\left( \frac{x}{\norm{x}}\right) \right)  \norm{x}^\beta \chi_{\varepsilon}(x) \widetilde{p}_t(\ud x)\Big\vert
\leq 
C\varepsilon \int_{\RR^d} \norm{x}^\beta \chi_{\varepsilon}(x) \widetilde{p}_t(\ud x) .
\end{align*}
To finish the proof we need the following observation: the measures $\widetilde{p}_t(\ud x)$, as $t$ goes to zero, converge weakly to the measure $\widetilde{p}_{\eta}$ which is uniquely determined by the formula
\begin{align*}
e^{-\int_{\RR^d} \left( 1-\cos \sprod{\xi}{y}\right) \eta (\ud y)} = \int_{\RR^d} e^{i\sprod{\xi}{y}}\widetilde{p}_{\eta}(\ud y).
\end{align*}
We work with characteristic functions and, since $\widetilde{p}_t(\ud x)$ is the distribution of the random variable $\psiI X_t$, it  suffices to show that
\begin{align*}
\lim_{t\to 0^+} t\psi \left( \psiI \xi \right) &= \int_{\RR^d} \left( 1-\cos \sprod{\xi}{y}\right) \eta (\ud y).
\end{align*}
We set $s = 1/\psiI$. Since $V\in \RInf$ with $\alpha >0$, we can investigate the above limit at zero along the new variable $s$ for which we have $t \sim 1/V(1/s)$, as $t$ tends to 0. Then in terms of the new variable we obtain
\begin{align*}
\frac{1}{V(1/s)}\, \psi \left( s^{-1}\xi \right) = \int_{\RR^d} \left( 1-\cos \sprod{\xi}{y}\right) \frac{\nu _s(\ud y)}{V(1/s)}
&= \int_{\RR^d} \left( 1- \chi_{\varepsilon}(y)\right)\left( 1-\cos \sprod{\xi}{y}\right) \frac{\nu _s(\ud y)}{V(1/s)}\\
&\quad +  \int_{\RR^d} \chi_{\varepsilon}(y) \left( 1-\cos \sprod{\xi}{y}\right)  \frac{\nu _s(\ud y)}{V(1/s)},
\end{align*}
where $\nu_s(G) = \nu (sG)$, for any Borel set $G\subset \RR^d$. Condition \eqref{Mult_Reg_Cond_Levy_meas} forces that the last integral converges, as $s$ goes to zero, to $\int \chi_\varepsilon (x) \left( 1-\cos \sprod{\xi}{y}\right) \eta (\ud y)$. Thus we are left to prove that the first integral approaches zero. We have
\begin{align*}
\int_{\RR^d} \left( 1- \chi_{\varepsilon}(y)\right)\left( 1-\cos \sprod{\xi}{y}\right) \frac{\nu _s(\ud y)}{V(1/s)}
\leq C \left( \int_{B_{\varepsilon }}+\int_{B_{1/\varepsilon }^c }\right) \left( 1-\cos \sprod{\xi}{y}\right) \frac{\nu _s(\ud y)}{V(1/s)}.
\end{align*}
For the second integral we apply Potter bounds: for $s$ small enough we have
\begin{align*}
\left\vert \int_{B_{1/\varepsilon }^c} \left( 1-\cos \sprod{\xi}{y}\right) \frac{\nu _s(\ud y)}{V(1/s)} \right\vert
\leq \frac{2}{V(1/s)} \nu_s(B_{1/\varepsilon}^c) =  \frac{2}{V(1/s)}V(\varepsilon /s)\leq  C\varepsilon ^{\alpha /2}.
\end{align*}
The first integral is bounded by
\begin{align*}
\int_{B_\varepsilon}\norm{\xi}^2\norm{y}^2\frac{\nu _s(\ud y)}{V(1/s)} = \frac{\norm{\xi}^2 \varepsilon ^2}{V(1/s)}\int_{B_{s\varepsilon }}\frac{\norm{y}^2}{(s\varepsilon)^2}\nu (\ud y)
\leq \norm{\xi}^2 \varepsilon ^2\frac{h(s\varepsilon)}{V(1/s)}\leq \norm{\xi}^2\varepsilon ^{2-\alpha},
\end{align*}
where in the last inequality we used the fact that for $s< 1$,
$h(s)\asymp V(1/s)$.
To obtain these two inequalities we write
\begin{align*}
h(s) &= \nu (B_s^c)+\frac{1}{s^2}\int_{B_s}\norm{y}^2\nu (\ud y) = V(1/s) + \frac{2}{s^2}\int_{B_s}\int_0^{\norm{y}}u\, \ud u\,  \nu (\ud y)  \\
&\leq V(1/s) + \frac{2}{s^2}\int_0^s u\, \nu(B_u^c)\ud u =  V(1/s) + \frac{2}{s^2}\int_0^s u\, V(1/u)\ud u.
\end{align*} 
The function $V(s)= \nu (B_{1/s}^c)$ is regularly varying at infinity of index $\alpha \in (\beta,2)$ and thus
the Karamata's theorem \cite[Section 1.6]{bgt} implies that the last integral behaves like $s^2V(1/s)$, as $s$ goes to zero, and we conclude the result. 
\end{proof}
\bibliographystyle{plain}

\end{document}